\documentclass[reqno,a4paper,12pt,russian,english]{amsart}
\usepackage{amsmath,amsthm,amsfonts,amssymb}
\usepackage{verbatim, graphicx, color, ifthen, enumitem}
\usepackage[T1]{fontenc}
\usepackage [applemac,koi8-r,latin9] {inputenc}
\usepackage{upquote,bbm}
\usepackage{babel}
\definecolor{myblue}{rgb}{0.09,0.32,0.44} 
\usepackage[pdfborder={0 0 0},pdfborderstyle={},colorlinks=true,linkcolor=myblue,citecolor=myblue,urlcolor=blue]{hyperref}

\allowdisplaybreaks

\let\oldmarginpar\marginpar
\renewcommand\marginpar[1]{\-\oldmarginpar[\raggedleft\footnotesize #1]%
{\raggedright\footnotesize #1}}
\theoremstyle{plain}

\newtheorem*{thm*}{Theorem}

\def\clap#1{\hbox to 0pt{\hss#1\hss}}

\usepackage{calrsfs}
\usepackage{graphicx}
\usepackage{color}
\usepackage{perpage}

\gdef\SetFigFontNFSS#1#2#3#4#5{} 

\theoremstyle{remark}
\newtheorem*{qst*}{Question}

\newtheorem*{claim*}{Claim}
\newtheorem{theorem}{Theorem}
\newtheorem{lemma}{Lemma}

\newtheorem{corollary}{Corollary}
\theoremstyle{definition}
\newtheorem{definition}{Definition}

\theoremstyle{remark}

\newtheorem*{remark*}{Remark}

\newcommand{\Z}{\mathbb{Z}}

\newcommand{\R}{\mathbb{R}}

\newcommand{\CC}{\mathbb{C}}

\def\L{\Lambda}
\def\l{\lambda}

\def\Z{\mathbb{Z}}

\def\R{\mathbb{R}}

\def\1{\mathbbm{1}}

\def\d{\delta}

\begin{document}

\title[Combining Riesz bases in $\R^d$]{Combining Riesz bases in $\R^d$}
\author{Gady Kozma and Shahaf Nitzan}

\begin{abstract}

We prove that every finite union of rectangles in $\R^d$ admits a Riesz basis of exponentials.
\end{abstract}
\maketitle
\section{introduction}

Orthogonal bases are used throughout mathematics and its applications. However, in many settings such
bases are not easy to come by. For example, even the union of as few as two
disjoint intervals in $\R$ may not admit an orthogonal basis of exponentials,
$e(\L):=\{e^{i\langle\lambda,t\rangle}\}_{\lambda\in\Lambda}$. This example should be contrasted with the case of a single interval, where the exponential orthogonal basis plays a fundamental role.

Among the systems which may be considered as replacements for orthogonal bases, Riesz bases are the best possible: They are the image of orthogonal bases under a bounded invertible operator and therefore \textit{preserve most of their qualities}. In particular, if $e(\L)$ is a Riesz basis over some set $S\subset\R$ then every $f\in L^2(S)$ can be decomposed into a series $f=\sum a_{\l}e^{2\pi i \l t}$ in a unique and stable way. 

Our understanding of the existence of Riesz bases
of exponentials is still lacking. On the one hand, there are relatively
few examples in which it is known how to construct a Riesz basis of
exponentials. For example, in two dimensions, we do not know how to
construct such a basis for either a ball or a triangle, nor even have
a reasonable candidate to be such a basis (it is known that neither
set supports an orthogonal basis of exponentials, \cite{F74, IKT03}).
Some constructions of Riesz bases (e.g.\ for polytopes with some arithmetic
constraints) can be found in \cite{GL14} and references within.
On the other hand, we know of no example of a set $S$ of positive
measure for which a Riesz basis of exponentials can be shown not to
exist.

In \cite{KozmaNitzan} we proved the following.
\begin{theorem}\label{result for intervals:p1}
Let $S\subset\mathbb{R}$ be a finite union of intervals. Then there
exists a set $\Lambda\subset\mathbb{R}$ such that the family
$e(\Lambda):=\{e^{2\pi i\lambda t}\}_{\lambda\in\Lambda}$
is a Riesz basis in $L^{2}(S)$. Moreover, if $S\subset [0,1]$ then $\L$ may be chosen to satisfy $\L\subset \Z$.
\end{theorem}

In this paper we extend this result to higher dimensions in the following way.
\begin{theorem}\label{mainresult:p2}
Let $S\subset\mathbb{R}^d$ be a finite union of rectangles with edges parallel to the axes. Then there
exists a set $\Sigma\subset\mathbb{R}^d$ such that the family $e(\Sigma)$
is a Riesz basis in $L^{2}(S)$. Moreover, if $S\subset [0,1]^d$ then $\Sigma$ may be chosen to satisfy $\Sigma\subset \Z^d$.
\end{theorem}

We take from \cite{KozmaNitzan} the following basic principle. Suppose
you try to construct a Riesz basis by combining Riesz bases of simpler
sets. If the
most natural candidate for a construction of a Riesz basis does not
work, try instead a construction that involves first taking unions and
intersections of your simpler sets, and then combining their Riesz
bases. Take, for example, in the setting of
theorem \ref{result for intervals:p1}, the case where $S=I\cup J$ with
$I\subset [0,1/2]$ and $J\subset [1/2,1]$. Then the most natural
candidate for a Riesz basis might be to take a Riesz basis for $I$ and
a Riesz basis for $J$ and union them. This does not work, but it turns
out that taking Riesz bases for $I\cup (J-1/2)$ and
for $I\cap (J-1/2)$ and taking a union works (under certain conditions). Here we need to
construct a Riesz basis for a union of products, say $\bigcup
X_i\times Y_i$. The natural candidate is to take Riesz bases $\Xi_i$ for $X_i$
and $\Psi_i$ for $Y_i$, and hope that $\bigcup \Xi_i\times\Psi_i$ is a
Riesz basis for $\bigcup X_i\times Y_i$. This does not work. The
correct ``union and intersection version'' is the following
lemma. Denote $Y_{\ge n}=\bigcup_{k=n}^L Y_n$.

\begin{lemma}\label{simple case:p2}
Let $X_1,\dotsc,X_L\subset [0,1]^{a}$ be some sets and let
$Y_1,\dotsc,Y_L\subset [0,1]^b$ be pairwise disjoint sets. Assume $\Xi_{1}\subset \dotsb\subset \Xi_{L}\subset\Z^a$ satisfy
that $e(\Xi_{n})$ is a Riesz
basis for $X_{n}$. Assume
further that $\Psi_{\ge 1}\supset\dotsb\supset \Psi_{\ge L}$ are
subsets of $\Z^b$ such that
that $e(\Psi_{\ge n})$ is a Riesz
basis of $Y_{\ge n}$. Define
\[
\Sigma:=\bigcup_{n=1}^L \Xi_{n}\times \Psi_{\ge n}.
\]
Then $e(\Sigma)$ is a Riesz basis for $\bigcup_{n=1}^L X_n\times Y_n$.
\end{lemma}

To get a feeling for the condition $\Xi_1\subset\dotsb\subset \Xi_L$
(which in particular means that the $X_n$ must have increasing sizes
for the lemma to have any hope of being applicable) one should first note that without this condition $\Sigma$ might not
even have the right density to be a Riesz basis (see
\cite{L67a,L67b,NO} for Landau's theorem, explaining the role of density). The definition of  $\Sigma$ can be
reorganized in two other ways which emphasize the issue of density (in
particular as a union of disjoint sets). The first is
\[
\Sigma = \bigcup_{n=1}^L \Xi_n\times (\Psi_{\ge n}\setminus \Psi_{\ge n+1})
\]
(where $\Psi_{\ge L+1}:=\emptyset$). This version has the mnemonic property of being almost a ``union of
products of Riesz bases'' except, of course, we are not requiring from
$\Psi_{\ge j}\setminus \Psi_{\ge j+1}$ to be a Riesz basis for
$Y_j$. The other version is
\[
\Sigma = \bigcup_{n=1}^L (\Xi_{n}\setminus\Xi_{n-1})\times
(\Psi_{\ge n})
\]
(where $\Xi_0:=\emptyset$). This version will be used in
the proof (\S \ref{sec:mainlemma} below).

More remarks on the relation with \cite{KozmaNitzan} will be given after the
proof, in \S 6.

\section{preliminaries}
\subsection{Systems of vectors in Hilbert spaces}

Let $H$ be a separable Hilbert space. A system of vectors $\{f_n\}\subseteq H$ is called a \textit{Riesz basis} if it is the image, under a bounded invertible operator, of an orthonormal basis. This means that $\{f_n\}$ is a Riesz basis if and only if it is complete in $H$ and satisfies the following inequality for all sequences $\{a_n\}\in l^2$,
\begin{equation}\label{riesz sequence}
c\sum|a_n|^2\leq \|\sum a_nf_n\|^2\leq C\sum |a_n|^2,
\end{equation}
where $c$ and $C$ are some positive constants which depend on the
system $f_n$ but not on the $a_n$. A system
$\{f_n\}\subseteq H$ which satisfies condition (\ref{riesz sequence}),
but is not necessarily complete, is called a \textit{Riesz sequence}.

A simple duality argument shows that $\{f_n\}$ is a Riesz basis if and only if it is minimal (i.e.\ no vector from the system lies in the closed span of the rest) and satisfies the following inequality for every $f\in H$,
\begin{equation}\label{frame}
c\|f\|^2\leq\sum |\langle f, f_n\rangle|^2\leq C\|f\|^2
\end{equation}
where $c$ and $C$ are some positive constants (in fact, the same constants as in (\ref{riesz sequence})). A system $\{f_n\}\subseteq H$ which satisfies condition (\ref{frame}), but is not necessarily minimal, is called a \textit{frame}.

In particular, this discussion implies the following:
\begin{lemma}\label{rb is rs and frame}
A system of vectors in a Hilbert space is a Riesz basis if and only if it is both a Riesz sequence and a frame.
\end{lemma}

In this paper we are interested in frames, Riesz sequences and Riesz
bases for $L^2(X)$ of the form $e(\Xi)$. Often we will be lax and
simply say that $\Xi$ is a frame, Riesz sequence or Riesz basis for
$X$. An important property of such sets is the complementation
property:
\begin{lemma}\label{lem:complm}A $\Xi\subset\mathbb{Z}^d$ is a frame over
  an $X\subset[0,1]^d$ if and only if $\mathbb{Z}^d\setminus\Xi$ is a Riesz
  sequence over $[0,1]^d\setminus X$.
\end{lemma}
Lemma \ref{lem:complm} follows from the following general fact:
\begin{lemma}\label{lem:Hilbertcomplm}Let $H$ be a separable Hilbert
  space and let $\{e_n\}_{n\in I}$ be an orthonormal basis in $H$. Let
  $L\subset H$ be a closed subspace of $H$ and let $L^\perp$ be its
  orthogonal complement. Denote by $P$ the orthogonal projection to
  $L$ and by $P^\perp$ the orthogonal projection to $L^\perp$. Then
  for a subset $\Xi\subset I$ we have that $\{Pe_n\}_{n\in \Xi}$ is a
  frame in $L$ if and only if $\{P^\perp e_n\}_{n\in I\setminus\Xi}$ is
  a Riesz sequence in $L^\perp$.
\end{lemma}
See Matei and Meyer \cite[Proposition 2.1]{MM09} for a proof.

\ifnum1=2
\subsection{Some corollaries of theorem \ref{result for intervals:p1}}

The following lemma was a main ingredient in the proof of theorem \ref{result for intervals:p1}, its proof can be found in \cite{KozmaNitzan}.

\begin{lemma}\label{lem:main:p1}
Fix a positive integer $N$. Given a set $S\subset [0,1]$, define
\begin{align}
&S_n=\Big\{t\in\Big[0,\frac1N\Big]:t+\frac jN\in S\textrm{ for exactly $n$
  values of $j\in\{0,\dotsc,N-1\}$}\Big\}\nonumber\\
&S_{\ge n}=\bigcup_{k=n}^N S_k\label{eq:defAgen}
\end{align}
If there exist $\L_1,\dotsc,\L_N\subseteq N\Z$ such that
the system $E({\L_n})$ is a Riesz basis in $L^2(S_{\ge n})$, then the system $E(\L)$, where
\[
\L=\bigcup_{j=1}^{N}(\L_j+j),
\]
is a Riesz basis in $L^2(S)$.
\end{lemma}

Combining this Lemma with theorem \ref{result for intervals:p1} we obtain the following corollaries, which were first stated in \cite{KozmaNitzan}.

\begin{corollary}\label{corr1:p2}
Let $S\subseteq [0,1]$ be a union of $L$ intervals and fix $N>0$. Assume that $m/N\leq|S|<(m+1)/N$ where $m$ is a positive integer. Then $\L$ of theorem 1 can be chosen to satisfy
    \[
    \cup_{j=0}^{m-2L-1}(N\Z+j)\subseteq \L \subseteq \cup_{j=0}^{m+2L}(N\Z+j).
    \]
\end{corollary}
\begin{proof}
Divide $[0,1]$ into the intervals $[j/N,(j+1)/N]$ and note that, since $m/N\leq|S|<(m+1)/N$ and $S$ is a union of $L$ intervals, at least $m-2L$ of the intervals
$[j/N,(j+1)/N]$ belong to $S$ and no more then $m+2L+1$ of them intersect $S$. This implies that among the corresponding sets $S_{\ge n}$ ate least $m-2L$ are equal to $[0,1/N]$ and no more then $m+2L$ are non--empty. By applying Theorem \ref{result for intervals:p1}, all these sets have a Riesz basis with frequencies from $N\Z$. Lemma \label{lem:main:prev} now implies the result.
\end{proof}
\begin{corollary}\label{corr2:p1}
 Let $S_1\subset S_2\subset\dotsb\subset S_K\subseteq [0,1]$ be a family of sets, all of which are finite unions of intervals. Then there exist sequences $\L_1\subset \L_2\subset\dotsb\subset \L_K\subset \Z$ such that $E(\L_k)$ is a Riesz basis over $S_k$.
\end{corollary}
\begin{proof}
 Denote the number of intervals in each set $S_k$ by $L_k$. Choose $N$ so large that for each $k$ we have $m_k/N\leq|S_k|<(m_k+1)/N$ and $m_k+2L_k+1\leq m_{k+1}+2L_{k+1}$. The statement now follows from corollary (\ref{corr1:p2}).
\end{proof}

We end this section with another lemma from \cite{KozmaNitzan}. The formulation here was not explicitly stated in \cite{KozmaNitzan}, but it is an immediate consequence of claim 3 and lemma 4 there.
\begin{lemma}\label{ichsa}
Let $S\subset[0,1]$ be a union of $L$ intervals and $N$ be a positive integer. Then, the sets $S_{\geq n}$ defined in lemma \ref{lem:main:p2} are all unions of at most $L$ intervals (when considered cyclically). Moreover, there exist infinitely many $N$ for which all these sets are unions of at most $L-1$ intervals (when considered cyclically).
\end{lemma}
\fi
\section{Proof of the main lemma}\label{sec:mainlemma}
In this section we prove lemma \ref{simple case:p2} which is the main
component of the proof of Theorem \ref{result for intervals:p1}. Recall
that we are given sets $X_j\subset \R^a$ and $Y_j\subset \R^b$ and
corresponding $\Xi_{j}\subset\Z^a,\Psi_{\ge j}\subset \Z^b$ and we wish to show
for $\Sigma=\bigcup (\Xi_{j}\setminus\Xi_{j+1})\times\Psi_{\ge j}$
that $e(\Sigma)$ is a Riesz basis for $S=\bigcup X_i\times Y_i$.
%
We will use Lemma \ref{rb is rs and frame} and show that $e(\Sigma)$ is
both a frame and a Riesz sequence for $L^2(S)$.

Throughout the proof we denote by
$(x,y):=(x_1,\dotsc,x_a,y_1,...,y_b)$ a point in $[0,1]^{a+b}$ and by
$(\xi,\psi):=(\xi_1,\dotsc,\xi_a,\psi_1,...,\psi_b)$ a point in
$\Z^{a+b}$. We denote by $e_{(\xi,\psi)}$ the function $e^{2\pi i
  \langle (\xi,\psi),(x,y)\rangle}$.

\subsection*{Frame}
To show that $e(\Sigma)$ is a frame in
$L^2(S)$ we need to show that for any $f\in L^2(S)$
\[
\sum_{(\xi,\psi)\in\Sigma}|\langle f, e_{(\xi,\psi)}\rangle|^2>c_1||f||^2
\]
(the right inequality in the definition of a frame, (\ref{frame}), is
satisfied because $S\subset[0,1]^{a+b}$ and $\Sigma\subset\Z^{a+b}$). For $k\in\{1,\dotsc,L\}$, denote by $f_k$ the restriction of $f$ to
$X_k\times Y_k$.
It is enough to show that for every $n=1,\dotsc,L$ we have
\begin{equation}\label{what we need for frame:p2}
\sum_{(\xi,\psi)\in\Sigma}|\langle f, e_{(\xi,\psi)}\rangle|^2\geq c_2\|f_n\|^2-\sum_{k=1}^{n-1}\|f_k\|^2,
\end{equation}
where $c_2$ is a positive constant, not depending on
$f$. 
Indeed, the inequalities in (\ref{what we need for frame:p2}) imply
that for any sequence of positive numbers $\{\delta_n\}_{n=1}^N$ with
$\sum\delta_n=1$ we have 
\begin{align*}
\sum_{(\xi,\psi)\in\Sigma}|\langle f, e_{(\xi,\psi)}\rangle|^2&=
\sum_{n=1}^L\delta_n\sum_{(\xi,\psi)\in\Sigma}|\langle f, e_{(\xi,\psi)}\rangle|^2\\
&\stackrel{\textrm{(\ref{what we need for frame:p2})}}{\geq}
 \sum_{n=1}^L\delta_n\Big(c_2\|f_n\|^2-\sum_{k=1}^{n-1}\|f_k\|^2\Big)
= \sum_{n=1}^L\Big(c_2\delta_n-\sum_{k=n+1}^{L}\delta_k\Big)\|f_n\|^2.
\end{align*}
We get that, if the sequence $\{\delta_n\}$ satisfies
\[
\delta_n>\frac 2{c_2}\sum_{k=n+1}^{L}\delta_k,\qquad\forall n\in\{1,\dotsc,L\}
\]
(essentially it needs to decrease exponentially), then for $c_1=\frac
12 c_2\min\delta_n$
\[
\sum_{(\xi,\psi)\in\Sigma}|\langle f, e_{(\xi,\psi)}\rangle|^2
\geq c_1\sum_{n=1}^L\|f_n\|^2=c_1\|f\|^2.
\]
as needed.

Hence we need to show (\ref{what we need for frame:p2}). Fix therefore
some $n\in\{1,\dotsc,L\}$ until the end of the proof. Now, for any $x,y\in\CC$, $|x+y|^2\ge \frac12
|x|^2-|y|^2$. So,
\[
|\langle f,e_{(\xi,\psi)}\rangle|^2\ge
\frac12\Big|\Big\langle\sum_{k=n}^Lf_k,e_{(\xi,\psi)}\Big\rangle\Big|^2 -
\Big|\Big\langle\sum_{k=1}^{n-1}f_k,e_{(\xi,\psi)}\Big\rangle\Big|^2.
\]
For brevity denote $f_{\ge n}=\sum_{k=n}^Lf_k$. Summing over all
$(\xi,\psi)$ in $\Sigma$ gives
\begin{align*}
\sum_{(\xi,\psi)\in\Sigma}|\langle f, e_{(\xi,\psi)}\rangle|^2
&\ge \frac{1}{2}\sum_{(\xi,\psi)\in\Sigma}|\langle f_{\ge n},e_{(\xi,\psi)}\rangle|^2-
  \sum_{(\xi,\psi)\in\Sigma}\Big|\Big\langle \sum_{k=1}^{n-1}f_k, e_{(\xi,\psi)}\Big\rangle\Big|^2\\
&\stackrel{\textrm{\clap{$(*)$}}}{\geq}
  \frac{1}{2}\sum_{(\xi,\psi)\in\Sigma}|\langle f_{\ge n}, e_{(\xi,\psi)}\rangle|^2-
  \Big\|\sum_{k=1}^{n-1}f_k\Big\|^2\\
&\stackrel{\textrm{\clap{$(**)$}}}{=}
  \frac{1}{2}\sum_{(\xi,\psi)\in\Sigma}|\langle f_{\ge n}, e_{(\xi,\psi)}\rangle|^2-
  \sum_{k=1}^{n-1}\|f_k\|^2
\end{align*}
where $(*)$ is because $\Sigma\subset\Z^d$ and $(**)$ since $f_k$ have disjoint supports. Hence, to obtain (\ref{what we need for frame:p2}) it remains to show that
\begin{equation}\label{more of what we need for frame:p2l1}
\sum_{(\xi,\psi)\in\Sigma}|\langle f_{\ge n}, e_{(\xi,\psi)}\rangle|^2\geq   c\|f_n\|^2
\end{equation}
where $c$ is a positive constant not depending on $f$.

Fix some $\xi\in\Xi_n$ and consider the function of $b$ variables
\[
F(y)=\int_{[0,1]^a} f_{\geq n}(x,y)e^{-2\pi i \langle\xi,x\rangle}\,dx
\]
and note that it
is supported on $Y_{\ge n}$. Since ${\Psi}_{\ge n}$ is a Riesz basis
for this set we have
\begin{align}
\sum_{\psi\in\Psi_{\ge n}}|\langle f_{\ge n},e_{(\xi,\psi)}\rangle|^2
&=\sum_{\psi\in\Psi_{\ge n}}\Big|\int_{[0,1]^{a+b}}f_{\ge
  n}(x,y)\overline{e_{(\xi,\psi)}(x,y)}\,dx\,dy\Big|^2\nonumber\\
&=\sum_{\psi\in\Psi_{\ge
    n}}\Big|\int_{[0,1]^{b}}F(y)\overline{e_{\psi}(y)}\,dy\Big|^2\nonumber\\
&\ge c\int_{Y_{\ge n}}|F(y)|^2\,dy\ge c\int_{Y_n}|F(y)|^2\,dy \nonumber\\
&=c\int_{Y_n}\Big|\int_{X_n}f_n(x,y)e^{-2\pi i \langle\xi,x\rangle}\,dx\Big|^2\label{eq:sumoverXi}
\end{align}
where the last equality follows from the fact that when $y\in Y_n$ we have $f_{\geq n}(x,y)=f_{ n}(x,y)$ and this function, as a function of $x$, is supported on $X_n$.

We now sum this over $\xi\in\Xi_n$. Recall that
$\Xi_{n}\times\Psi_{\ge n}\subset\Sigma$ and that
$e(\Xi_n)$ is a Riesz basis for $L^2(X_n)$. We get
\begin{align*}
\sum_{(\xi,\psi)\in\Sigma}|\langle f_{\ge n}, e_{(\xi,\psi)}\rangle|^2
& \geq
\sum_{\xi\in\Xi_{n},\psi\in\Psi_{\ge n}}|\langle f_{\ge n}, e_{(\xi,\psi)}\rangle|^2\\
&\stackrel{\textrm{(\ref{eq:sumoverXi})}}{\ge}
c\sum_{\xi\in\Xi_n}\int_{Y_n}\Big|\int_{X_n} f_{ n}(x,y)e^{-2\pi i \langle\xi,x\rangle}\,dx\Big|^2dy\\
&=c\int_{Y_n}\sum_{\xi\in\Xi_n}\Big|\int_{X_n} f_{ n}(x,y)e^{-2\pi i\langle \xi,x\rangle}\,dx\Big|^2dy\\
&\ge c\int_{Y_n}\int_{X_n}|f_{ n}(x,y)|^2\,dx\,dy=
c\|f_n\|^2.
\end{align*}
Hence, (\ref{more of what we need for frame:p2l1}) holds and the system is a frame.

\subsection*{Riesz sequence} We now show that $e(\Sigma)$ is a Riesz
sequence in $L^2(S)$, i.e.\ that for any \emph{finitely supported} sequence $a_{(\xi,\psi)}\in l^2(\Sigma)$,
\[
\bigg\Vert\sum_{(\xi,\psi)\in\Sigma}a_{(\xi,\psi)} e_{(\xi,\psi)}\bigg\Vert^2_{L^2(S)}\ge
c\sum_{(\xi,\psi)\in\Sigma}|a_{(\xi,\psi)}|^2
\]
(again, the other inequality in (\ref{riesz sequence}) follows from
$S\subset[0,1]^d$ and $\Sigma\subset\Z^d$).
We apply a strategy similar to the one we used in the first
(``frame'') part, but
we decompose $\Sigma$ rather than $S$. Define therefore
$\Sigma_n=(\Xi_{n}\setminus\Xi_{n-1})\times\Psi_{\ge n}$. With this
definition a similar argument to the one used in the first part shows that it is enough to show that for every $n=1,\dotsc,L$ we have
\begin{equation}\label{what we need for r-s:p2l1}
\int_{S}\Big|\sum_{(\xi,\psi)\in\Sigma}a_{(\xi,\psi)}
e_{(\xi,\psi)}\Big|^2\geq c\sum_{(\xi,\psi)\in\Sigma_n}|a_{(\xi,\psi)}|^2-\sum_{j=n+1}^{L}\sum_{(\xi,\psi)\in\Sigma_j}|a_{(\xi,\psi)}|^2.
\end{equation}
To this end choose $n\in\{1,\dotsc,L\}$. We have,
\begin{align*}
\lefteqn{\int_{S}\Big|\sum_{(\xi,\psi)\in\Sigma}a_{(\xi,\psi)} e_{(\xi,\psi)}\Big|^2\,dx\,dy}\qquad\\&\geq
\frac{1}{2}\int_{S}\Big|\sum_{j=1}^n\sum_{(\xi,\psi)\in\Sigma_j}a_{(\xi,\psi)} e_{(\xi,\psi)}\Big|^2-\int_{S}\Big|\sum_{j=n+1}^{L}\sum_{(\xi,\psi)\in\Sigma_j}a_{(\xi,\psi)} e_{(\xi,\psi)}\Big|^2\\
&\ge\frac{1}{2}\int_{S}\Big|\sum_{j=1}^n\sum_{(\xi,\psi)\in\Sigma_j}a_{(\xi,\psi)} e_{(\xi,\psi)}\Big|^2-\sum_{j=n+1}^{L}\sum_{(\xi,\psi)\in\Sigma_j}|a_{(\xi,\psi)}|^2
\end{align*}
where the second inequality is due to $S\subset[0,1]^{a+b}$ and $\Sigma\subset\Z^{a+b}$. Denote for
brevity
\[
f=\1_S\cdot \sum_{j=1}^n \sum_{(\xi,\psi)\in\Sigma_j}a_{(\xi,\psi)} e_{(\xi,\psi)}
\]
and get that to prove (\ref{what we need for r-s:p2l1}) it remains to show that
\begin{equation}\label{more we need for r-s:p2l1}
\int_{S}|f(x,y)|^2\,dx\,dy\geq c\sum_{(\xi,\psi)\in\Sigma_n}|a_{(\xi,\psi)}|^2.
\end{equation}
Here is where the fact that $\Xi_{n}$ and $\Psi_{\ge n}$ are Riesz bases
will enter.

We first apply that $\Xi_n$ is a Riesz sequence over $X_k$ for all
$k\geq n$, specifically the left inequality in (\ref{riesz sequence}),
 and get, for any $y$,
\begin{align}
\int_{X_k}|f(x,y)|^2\,dx
&=\int_{X_k}\Big|\sum_{j=1}^n\sum_{\xi\in\Xi_j\setminus\Xi_{j-1}}
\Big(\sum_{\psi\in\Psi_{\ge j}}a_{(\xi,\psi)}e^{2\pi i \langle\psi,y\rangle}\Big)e^{2\pi i \langle\xi,x\rangle}\Big|^2dx\nonumber\\
\textrm{By (\ref{riesz sequence})}\qquad
&\geq c \sum_{j=1}^n\sum_{\xi\in\Xi_j\setminus\Xi_{j-1}}
\Big|\sum_{\psi\in\Psi_{\ge j}}a_{(\xi,\psi)}e^{2\pi i \langle\psi,y\rangle}\Big|^2\nonumber\\
\textrm{dropping terms}\qquad
& \geq
c\sum_{\xi\in\Xi_n\setminus\Xi_{n-1}}\Big|\sum_{\psi\in\Psi_{\ge n}}
a_{(\xi,\psi)}e^{2\pi i \langle\psi,y\rangle}\Big|^2\label{eq:onex}
\end{align}
Integrating over $y$ and using the fact that $\Psi_{\ge n}$
is a Riesz basis over $Y_{\ge n}$ we get
\begin{align*}
\int_{S}|f(x,y)|^2\,dx\,dy
&\geq \sum_{k=n}^L\int_{Y_k}\int_{X_k}|f(x,y)|^2\,dx\,dy\\
& \stackrel{\textrm{(\ref{eq:onex})}}{\geq}
c\sum_{k=n}^L\int_{Y_k}\sum_{\xi\in\Xi_n\setminus\Xi_{n-1}}
\big|\sum_{\psi\in\Psi_{\ge n}}a_{(\xi,\psi)}e^{2\pi i \langle\psi,y\rangle}\big|^2\,dy\\
&= c\sum_{\xi\in\Xi_n\setminus\Xi_{n-1}}\int_{Y_{\geq
    n}}\big|\sum_{\psi\in\Psi_{\ge n}}a_{(\xi,\psi)}e^{2\pi i \langle\psi,y\rangle}\big|^2\,dy\\
\textrm{Since $\Psi_{\ge n}$ is a Riesz basis}\qquad
&\geq c\sum_{\xi\in\Xi_n\setminus\Xi_{n-1}}\sum_{\psi\in\Psi_{\ge n}}|a_{(\xi,\psi)}|^2
\end{align*}
which asserts (\ref{more we need for r-s:p2l1}) and completes the
proof.\qed

\begin{remark*}
The ``frame'' and ``Riesz sequence'' parts are in fact independent in
the following sense. If $\Xi_1\subset\dotsb\subset\Xi_L$ and $\Psi_{\ge
  1}\supset\dotsb\supset\Psi_{\ge L}$ are only assumed to be frames,
then $\Sigma$ will be a frame; while if they are assumed to be Riesz
sequences then $\Sigma$ will be a Riesz sequence. In the next section
we will see that in another setting this remark allows to shorten the
proof.
\end{remark*}

\section{Folding}\label{sec:fold}
In this section we prove a version of the main lemma
of \cite{KozmaNitzan}. That lemma stated that if certain
``foldings'' of a set have Riesz bases, then one may
construct a Riesz basis for the original set too. The result here,
while stated in $d$-dimensions, is essentially one dimensional and we
will perform the same transformations performed in \cite{KozmaNitzan}
on the first coordinate only. The details are below. The proof is also
similar to the proof there, but
with a simplification suggested by A.\ Olevski\u\i.
Throughout this section we will denote either by
$(t,s):=(t,s_1,...,s_{d-1})$ or by $x$ a point in $[0,1]^d$; and by
$(\l,\d):=(\l,\d_1,...,\d_{d-1})$ or $\xi$ a point in $\Z^d$. 

Fix a positive integer $N$. Given a set $X\subset [0,1]^d$, define

  \begin{align}
X_n=&\Big\{t\in\Big[0,\frac1N\Big]\times[0,1]^{d-1}:\Big(t+\frac jN ,s\Big)\in X\\ &\textrm{ for exactly $n$
  values of $j\in\{0,\dotsc,N-1\}$}\Big\}\nonumber\\
X_{\ge n}=&\bigcup_{k=n}^N X_k\label{eq:defAgen}
\end{align}

\begin{lemma}\label{lem:main:p2}
If there exist $\Xi_1,\dotsc,\Xi_N\subseteq N\Z\times\Z^{d-1}$ such that
the system $e({\Xi_n})$ is a Riesz basis in $L^2(X_{\ge n})$, then the system $e(\Xi)$, where
\[
\Xi=\bigcup_{n=1}^{N}(\Xi_n+(n,0,\dotsc,0)),
\]
is a Riesz basis in $L^2(S)$.
\end{lemma}

Clearly, it is equivalent to prove the lemma under the assumptions
that
\begin{equation}\label{eq:defLam}
\Xi_n\subset (N\Z+n)\times\Z^{d-1}\qquad \Xi=\bigcup_{n=1}^N\Xi_n
\end{equation}
(but still requiring that $\Xi_n$ is a Riesz basis for $X_{\ge n}$,
recall that the property of being a Riesz basis is invariant to translations) which will make the notations a little shorter.

We will show that $e(\Xi)$ is a Riesz basis by showing that it is
both a frame and a Riesz sequence (recall lemma \ref{rb is rs and frame}). It turns out that to show that $\Xi$
is a frame it is enough that all $\Xi_j$  are frames. Let us state this
as a Lemma.

\begin{lemma}\label{lem:oned frame}
If $\Xi_n\subset (N\Z+n)\times\Z^{d-1}$ satisfy that
$e(\Xi_n)$ is a frame in $L^2(X_{\ge n})$ for all $n\in\{1,\dotsc,N\}$, then the system $e(\Xi)$
is a frame in $L^2(X)$, where $\Xi$ is given by \eqref{eq:defLam}.

Furthermore, the same holds if $\Xi_n\subset (N\Z-n)\times\Z^{d-1}$.
\end{lemma}

\begin{proof}
To show that $e(\Xi)$ is a frame in
$L^2(X)$ we need to show that for any $f\in L^2(X)$
\[
\sum_{\xi\in\Xi}|\langle f, e_{\xi}\rangle|^2>c_1||f||^2
\]
(the right inequality in the definition of a frame, (\ref{frame}), is
satisfied because $X\subset[0,1]^d$ and $\Xi\subset\Z^d$). For $n\in\{1,\dotsc,N\}$, denote by $f_n$ the restriction of $f$ to
\begin{equation}\label{eq:defAn}
B_n=\Big\{(t,s)\in X: \Big(t+\frac{j}{N},s\Big)\in X \textrm{ for exactly } n
\textrm{ integer } j{'s}\Big\}.
\end{equation}
($X_n$ is the ``folding'' of $B_n$ to $[0,\frac 1N]\times[0,1]^{d-1}$ i.e.\ cutting to $N$ pieces, translating each one to $[0,\frac 1N]\times[0,1]^{d-1}$ and taking a union). For brevity denote $f_{\ge n}=\sum_{k=n}^Nf_k$.
As in the proof of Lemma \ref{simple case:p2}, it is enough to show,
for every $n=1,\dotsc,N$,
that $\sum_\xi|\langle f,e_\xi\rangle|^2\ge
c\|f_n\|^2-C\sum_{k=1}^{n-1}\|f_k\|^2$. And, again as in the proof of
Lemma \ref{simple case:p2}, this can be reduced further to showing that
\begin{equation}\label{more of what we need for frame}
\sum_{\xi\in\Xi}|\langle f_{\ge n}, e_{\xi}\rangle|^2\geq   c\|f_n\|^2
\end{equation}
where $c$ is a positive constant not depending on $f$. The rest of the
proof only examines one $n$ at a time, so let us fix $n$ now.

 For any $(\l,\d)\in (N\Z+j)\times[0,1]^{d-1}$ we have
\begin{align}
\langle f_{\ge n},e_{(\l,\d)}\rangle
  &=\int_{[0,1]^{d-1}}\int_0^1 f_{\ge n}(t,s)\overline{e_{(\l,\d)}(t,s)}\,dt\,ds\nonumber\\
  &=\int_{[0,1]^{d-1}}\int_0^{1/N}\sum_{l=0}^{N-1}f_{\ge n}\Big(t+\frac
  lN,s\Big)e_{(-\l,-\d)}\Big(t+\frac lN,s\Big)\,dt\,ds \nonumber\\
&=\int_{[0,1]^{d-1}}\int_0^{1/N}h_j(t,s)e(-\l t-\langle \d,s\rangle)\,dt\,ds
  =\langle h_j,e_{(\l,\d)}\rangle.\label{eq:dumb}
\end{align}
where
\[
h_j(t,s)=\1_{X_{\ge n}}(t,s)\cdot \sum_{l=0}^{N-1}f_{\ge
  n}\Big(t+\frac lN,s\Big)q_j^l \qquad q_j=e\Big(-\frac jN\Big).
\]
Fix $j\leq n$. Since $e(\Xi_j)$ is a frame for $X_{\ge j}$ and since $h_j$
is supported on $X_{\ge n}\subset X_{\ge j}$ we have
\begin{equation}\label{eq:Riesz}
\sum_{(\l,\d)\in\Xi_j}|\langle f_{\ge n},e_{(\l,\d)}\rangle|^2
\stackrel{\textrm{(\ref{eq:dumb})}}{=}
\sum_{(\l,d)\in\Xi_j}|\langle h_j,e_{(\l,\d)}\rangle|^2
\ge c ||h_j||^2
\end{equation}
where $c$ is the frame constant of $\Xi_j$. In the ``furthermore''
clause of the lemma (where $\Xi_j\subset N\Z-j$) we define
$q_j=e(j/N)$ instead of $e(-j/N)$ and the calculation follows identically.

Summing over $j$ gives
\begin{align}\label{eq:Fh}
\sum_{\xi\in\Xi}|\langle f_{\ge n},e_{\xi}\rangle|^2&\ge
\sum_{j=1}^n\sum_{\xi\in\Xi_j}|\langle f_{\ge n},e_{\xi}\rangle|^2\ge\nonumber\\
&\stackrel{\textrm{(\ref{eq:Riesz})}}{\ge}
c\sum_{j=1}^n||h_j||^2\ge c\sum_{j=1}^n||h_j\cdot\1_{X_n}||^2.
\end{align}

For every particular $(t,s)\in X_n$ the values of $\{h_j(t,s)\}_j$ are given by
applying the $n\times N$ matrix $L=\{q_j^l\}_{j,l}$ to the vector
$\{f_{\ge n}(t+l/N,s)\}_l$. Now, $(t,s)\in X_n$ so exactly $n$ different values of this vector are non-zero.
Considering only these values we may think of $L$ as an
$n\times n$
Vandermonde matrix which is invertible because the numbers $q_j$
are different. Let $C$ be a bound for the norm of the inverse over all such $n\times n$ sub-matrices of $L$. We get
\[
\sum_{j=1}^n |h_j(t,s)|^2 \ge \frac 1C\sum_{l=0}^{N-1} \Big|f_{\ge n}\Big(t+\frac
lN,s\Big)\Big|^2
\]
which we integrate over $(t,s)\in X_n$ to get
\[
\sum_{j=1}^n ||h_j\cdot\1_{X_n}||^2
\ge c\sum_{l=0}^{N-1} \int_{X_n}\Big|f_{\ge n}\Big(t+\frac lN,s\Big)\Big|^2\,dt\,ds
=c||f_{n}||^2.
\]
With this we get (\ref{more of what we need for frame}) and therefore that
$\Xi$ is a frame.
\end{proof}
\begin{proof}[Proof of Lemma \ref{lem:main:p2}]
We apply Lemma \ref{lem:oned frame} twice. The first application is
straightforward with the same $X$ and $\Xi_n$ and we get that $\Xi$ is a frame. For
the second application, let $Y=[0,1]^d\setminus X$ and note that
$Y_{\ge n}=[0,1/N]\times [0,1]^{d-1}\setminus X_{\ge N+1-n}$ (for $Y_n$
the correspondence is not as nice as it is for $Y_{\ge n}$). Since $\Xi_{n}$ is a Riesz basis
for $X_{\ge n}$, in particular a Riesz sequence, by Lemma \ref{lem:Hilbertcomplm}
$(N\Z+1-n)\times\Z^{d-1}\setminus \Xi_{N+1-n}$ is a frame for $Y_{\ge
  n}$. We now apply Lemma
\ref{lem:oned frame} for $Y$ and the complements of $\Xi_n$ (we use
the ``furthermore'' clause to rearrange them in decreasing order) and get
that
\[
\bigcup_{n=1}^N (N\Z+1-n)\times \Z^{d-1}\setminus \Xi_{N+1-n}
\]
is a frame for $Y$ (we used here that a translation of a frame is also
a frame, to solve $+1$ problems). But this set is exactly $\Z^d\setminus
\Xi$ and using lemma \ref{lem:Hilbertcomplm} again we
get that $\Xi$ is a Riesz sequence for $X$, and we are done.
\end{proof}

We end this section with another lemma from \cite{KozmaNitzan}. It is
a consequence of claim 3 and lemma 4 there.
\begin{lemma}\label{ichsa}
Let $X\subset[0,1]$ be a union of $L$ intervals and $N$ be a positive
integer. Then, the sets $X_{\geq n}$ defined before Lemma
\ref{lem:main:p2} are all unions of at most $L$ intervals (when
considered cyclically). Moreover, there exist infinitely many $N$ for
which all these sets are unions of at most $L-1$ intervals (again,
when considered cyclically).
\end{lemma}
Here and below a ``cyclic interval'' is either an $[a,b]\subset
[0,1/N]$ for $a<b$ or a $[0,b]\cup[a,1/N]$ for $b<a$.
\section{Proof of theorem \ref{mainresult:p2}}

The proof follows by induction and, as is quite typical for inductive
proofs, we need to prove a stronger claim in order to make the
induction tick. We describe it in the following definition
\begin{definition}Let $X_1,X_2,\dotsc\subset[0,1]^d$. A \emph{coherent
    collection of Riesz bases} are $\Xi_i\subset \Z^d$ such that
  $e(\Xi_i)$ is a Riesz basis for $X_i$ and such that $X_i\subset X_j$
  implies that $\Xi_i\subset \Xi_j$.
\end{definition}
The ``stronger claim'' above is now
\begin{theorem}\label{generalmain:p2}
Any collection of sets, each of which is a union of rectangles with
edges parallel to the axes, has a coherent collection of Riesz bases.
\end{theorem}
The proof of the $d=1$ case will follow easily from the following lemma
\begin{lemma}\label{corr1:p2}
Let $X\subset [0,1]$ be a union of $L$ intervals and fix $N>0$. Assume
that $m/N\leq|X|<(m+1)/N$ where $m$ is a positive integer. Then there
exists a $\Xi$ with $e(\Xi)$ a Riesz basis in $L^2(X)$ such that
    \[
    \cup_{n=0}^{m-2L-1}(N\Z+n)\subseteq \L \subseteq \cup_{n=0}^{m+2L}(N\Z+n).
    \]
\end{lemma}
\begin{proof}
Divide $[0,1]$ into the intervals $[n/N,(n+1)/N]$ and note that, since $m/N\leq|X|<(m+1)/N$ and $X$ is a union of $L$ intervals, at least $m-2L$ of the intervals
$[n/N,(n+1)/N]$ belong to $X$ and no more then $m+2L+1$ of them
intersect $X$. We wish to apply Lemma
\ref{lem:main:p2} with this $N$, so examine the sets $X_{\ge n}$ from
the statement of the lemma. We get that among the $X_{\ge n}$ at least
$m-2L$ are equal to $[0,1/N]$ (so the corresponding $\Xi_n$ can, and
must be taken to be $N\Z$) and no more then $m+2L$ are non-empty (for
which the $\Xi_n$ must be taken empty). The remaining sets are finite
unions of intervals so we may apply Theorem
\ref{result for intervals:p1} to find Riesz bases for them with
frequencies from $N\Z$. Applying Lemma \ref{lem:main:p2} the
resulting basis $\Xi$ has the necessary property.
\end{proof}

\begin{proof}[Proof of Theorem \ref{generalmain:p2}]
As promised, the case $d=1$ follows directly from Lemma
\ref{corr1:p2}. Indeed, let $X_i$ be the unions of rectangles
(intervals in our case) for which we need to find a coherent collection of
Riesz bases. Let $L$ be
the maximum number of intervals in any $X_i$ and
take $N>4L/\min|X_j\setminus X_i|$, where the minimum is taken over
all $i$ and $j$ such that $X_i\subset X_j$.
Construct Riesz bases $\Xi_i$ for $L^2(X_i)$ using
Lemma \ref{corr1:p2} with this $N$. We get that the $\Xi_i$ are
automatically coherent as $X_i\subset X_j$ implies that, for any $k$,
if $\Xi_i\cap (N\Z+k)\ne\emptyset$ then necessarily $N\Z+k\subset\Xi_j$. This
finishes the case $d=1$.

We now move to the case $d>1$.

\subsection*{Step 1.} First, we prove the induction step in the case where
the intersection of each $X_i$ with each line parallel to the first
coordinate axis is an interval.
\begin{claim*}In this case, it is possible to find disjoint sets
$Y_j\subset [0,1]^{d-1}$ and intervals $I_{i,j}\subset[0,1]$ (possibly
  empty) such that each $X_i$ can be written as
\[
X_i=\bigcup_j I_{i,j} \times Y_j.
\]
Further, all $Y_j$ can be taken to be finite unions of rectangles.
\end{claim*}
The proof of this claim is simple (take a total refinement of
appropriate projections of parts of the $X_i$) and will be omitted.

Returning to the proof of Theorem \ref{generalmain:p2}, we first use the case $d=1$ already established to find a coherent
collection of Riesz bases $\L_{i,j}$ for the intervals $[0,|I_{i,j}|]$
i.e.\ for translations of $I_{i,j}$ so that their left side is at
0. Since the property of being a Riesz basis is translation invariant
we get that each $\L_{i,j}$ is a Riesz basis for $I_{i,j}$. In other
words, $\L_{i,j}$ is a collection of Riesz bases for $I_{i,j}$ with
the property that if $|I_{i,j}|\le|I_{i',j'}|$ then $\L_{i,j}\subseteq \L_{i',j'}$.

Next we apply the induction assumption for $d-1$ and get a coherent collection of
Riesz bases for all \emph{finite unions} of the $Y_j$. Denote, for each set
of indices $J$, $Y_J=\bigcup_{j\in J}Y_j$ and $\Psi_J\subset\Z^d$ the Riesz basis
over $Y_J$.

For each $i$ let $\sigma$ be the rearrangement of $I_{i,j}$ by
length, i.e.\ $\sigma$ is a permutation such that $|I_{i,\sigma(1)}|\le
|I_{i,\sigma(2)}|\le\dotsb$ and define
\[
\Xi_i=\bigcup_j \L_{i,\sigma(j)}\times \Psi_{\{\sigma(j),\sigma(j+1),\dotsc\}}.
\]
By Lemma \ref{simple case:p2} $\Xi_i$ is a Riesz basis for $X_i$. To
see coherency, let $i$ and $i'$ satisfy that $X_{i}\subset
X_{i'}$ and let $\sigma$ and $\sigma'$ be the corresponding
permutations. To shorten notations denote the different pieces of
$\Xi$ and $\Xi'$ by $A_j$ and $A_j'$ respectively i.e.
\[
A_j:=\L_{i,\sigma(j)}\times
\Psi_{\{\sigma(j),\sigma(j+1),\dotsc\}}
\qquad
A_j':=\L_{i',\sigma'(j)}\times
\Psi_{\{\sigma'(j),\sigma'(j+1),\dotsc\}}.
\]
We need to show that $\Xi_{i}\subset \Xi_{i'}$, and this will
follow once we show that for every $j$ there exists $k$ such that
$A_j\subset A_k'$. Fix therefore $j$ and examine the $j$ shortest
intervals for $X_{i'}$ i.e.\ $\sigma'(1),\dotsc,\sigma'(j)$. They
cannot be all in the set $\sigma(1),\dotsc,\sigma(j-1)$ so let $k$
be the first which is not in it i.e.
\[
k=\inf\big\{l\le j:
\sigma'(l)\not\in\{\sigma(1),\dotsc,\sigma(j-1)\}\big\}.
\]
The claim now follows easily. We first note that
\[
|I_{i',\sigma'(k)}|\stackrel{(*)}{\ge} |I_{i,\sigma'(k)}|
\stackrel{(**)}{\ge}|I_{i,\sigma(j)}|
\]
where $(*)$ is because $X_i\subset X_{i'}$ and $(**)$ is because
$\sigma'(k)$ is not in $\{\sigma(1),\dotsc,\sigma(j-1)\}$. Hence the coherency of the $\L$'s gives that
$\L_{i',\sigma'(k)}\supset\L_{i,\sigma(j)}$. On the other hand,
$\{\sigma'(1),\dotsc,\sigma'(k-1)\}\subset\{\sigma(1),\dotsc,\sigma(j-1)\}$
and taking complements gives
\[
\{\sigma'(k),\sigma'(k+1),\dotsc\}\supset\{\sigma(j),\sigma(j+1),\dotsc\}
\]
and the coherency of the $\Psi$'s gives that
$\Psi_{\{\sigma'(k),\sigma'(k+1),\dotsc\}}
\supset\Psi_{\{\sigma(j),\sigma(j+1),\dotsc\}}$. Together we get
$A_j\subset A_k'$, as required.

\subsection*{Step 2.} As in step 1, we find disjoint sets
$Y_j\subset[0,1]^{d-1}$ and $S_{i,j}\subset [0,1]$ (which are no longer
necessarily intervals, but are finite unions of intervals) such that
\[
X_i=\bigcup_jS_{i,j}\times Y_j.
\]
Let $M_{i,j}$ be the number of components of $S_{i,j}$. We argue by
induction on the vector $\{M_{i,j}\}$, with the case that all $M_{i,j}$
are either 0 or 1 given by step 1.

Let therefore $i_0$ and $j_0$ satisfy that $M_{i_0,j_0}\ge 2$. Recall the
notation $X_{\ge n}$ from \S\ref{sec:fold}, which was defined with
respect to some $N$ which does not appear in the notation. When we
apply it to sets which already have a subscript, like $S_{i,j}$, we will write
$S_{i,j,\ge n}$. By Lemma \ref{ichsa} we can find some $N$ such that the sets
$S_{i_0,j_0,\ge n}$ contain no more than $M_{i_0,j_0}-1$ intervals, for
all $n$. Since the operation $\cdot_{\ge n}$ examines only the first
coordinate, and because the $Y_j$ are disjoint, we have
\[
X_{i,\ge n}=\bigcup_j S_{i,j,\ge n}\times Y_j.
\]
Again by Lemma \ref{ichsa}, $S_{i,j,\ge n}$ has no more than $M_{i,j}$
components, for all $i$ and $j$. Therefore we may apply our induction
hypothesis to $X_{i,\ge n}$ (formally after stretching the first
coordinate by $N$) and get a coherent collection of Riesz
bases $\Xi_{i,n}$ in $N\Z\times\Z^{d-1}$. Define
\[
\Xi_i=\bigcup_{n=1}^N\Xi_{i,n}+(n,0,\dotsc,0).
\]
and get from Lemma \ref{lem:main:p2} that $e(\Xi_i)$ is a Riesz basis
for $L^2(X_i)$. Since $X_i\subset X_j$ implies that
$X_{i,\ge  n}\subset X_{j,\ge n}$, we get that the $\Xi_i$ are coherent,
finishing step 2 and the proof of the theorem.
\end{proof}

\section{Remarks on the proof}

The main ingredient in the proof of theorem \ref{result for
  intervals:p1} from \cite{KozmaNitzan} was the one-dimensional case
of lemma \ref{lem:main:p2}. Examining its proof it is natural to
wonder whether it could have been generalized directly to prove the
$d$-dimensional result by folding in all dimensions simultaneously. As
far as we can see, this is not possible. The proof of lemma
\ref{lem:main:p2} relies on the fact that for any choice of $n$
columns in the $N\times N$ Fourier matrix the first $n$ rows will
give, universally, an $n\times n$ invertible matrix, as it is a
Vandermonde matrix. This is not the case for the analog of the Fourier
matrix in higher dimensions, no such ``universal'' choice of rows
exists, as can be checked directly for the $4\times 4$ matrix of the
Fourier transform the group $(\Z/2)^2$. 

\subsection*{Acknowledgements} We thank A.\ Olevski\u\i{} for showing us
how to deduce the Riesz sequence part from the frame part in the proof
of Lemma \ref{lem:main:p2}. Work partly supported by the
Israel Science Foundation and the Jesselson Foundation.

\begin{flushright}
\footnotesize
Gady Kozma\\
\nolinkurl{gady.kozma@weizmann.ac.il}\\
The Weizmann Institute of Science, Rehovot, Israel

\medskip

Shahaf Nitzan\\
\nolinkurl{shahaf.n.h@gmail.com}\\
Kent State University, Kent, OH, USA

\end{flushright}


\begin{thebibliography}{HNP80}




\bibitem{F74}Bent Fuglede, \emph{Commuting self-adjoint partial
differential operators and a group theoretic problem}, J. Funct. Anal.
16:1 (1974), 101--121. \href{http://www.sciencedirect.com/science/article/pii/002212367490072X}{\nolinkurl{sciencedirect.com/00221236}}


\bibitem{GL14}Sigrid Grepstad and Nir Lev, \emph{Multi-tiling and Riesz bases}.
Adv. Math. 252 (2014), 1--6. Available at:
\href{http://www.sciencedirect.com/science/article/pii/S0001870813003940}{\nolinkurl{sciencedirect.com/S0001870813003940}}, \href{http://arxiv.org/abs/1212.4679}{\nolinkurl{arXiv:1212.4679}}


\bibitem{IKT03}Alex Iosevich, Nets Katz and Terence Tao, \emph{The
Fuglede spectral conjecture holds for convex planar domains}. Math.
Res. Lett. 10:5 (2003), 559--569.
\href{http://www.intlpress.com/_newsite/site/pub/pages/journals/items/mrl/content/vols/0010/0005/00019975/index.php}{\nolinkurl{intlpress/00019975}}

\bibitem{IK}Alex Iosevich and Mihalis N. Kolountzakis, {\em
  Periodicity of the spectrum in dimension one}. Analysis and PDE 6:4
  (2013), 819--827. Available at:
  \href{http://msp.org/apde/2013/6-4/p03.xhtml}{\nolinkurl{msp.org}},
  \href{http://arxiv.org/abs/1108.5689}{\nolinkurl{arXiv:1108.5689}}







\bibitem{KozmaNitzan}Gady Kozma and Shahaf Nitzan, \emph{Combining
  Riesz bases}. Invent. Math. 199:1 (2015), 267--285. Available at:
  \href{http://dx.doi.org/10.1007/s00222-014-0522-3}{\nolinkurl{springer.com/s00222-014-0522-3}}, \href{http://arxiv.org/abs/1210.6383}{\nolinkurl{arXiv:1210.6383}}

\bibitem{L67a} Henry J. Landau, \emph{Necessary density conditions
  for sampling and interpolation of certain entire functions}. Acta
  Math. 117 (1967) 37--52. \href{http://www.springerlink.com/content/22h1h1514x501740/}{\nolinkurl{springerlink.com/22h1h1514x}}

\bibitem{L67b} \bysame, \emph{Sampling, data transmission,
  and the Nyquist rate}. Proc. IEEE 55:10 (1967),
  1701--1706. Available at: \href{http://dx.doi.org/10.1109/PROC.1967.5962}{\nolinkurl{ieee.org/PROC.1967.5962}}






\bibitem{MM09}Basarab Matei and Yves Meyer, \emph{A variant of
  compressed sensing}. Rev. Mat. Iberoamericana 25:2 (2009),
  669--692. Available at:
  \href{http://www.ems-ph.org/journals/show_abstract.php?issn=0213-2230&vol=25&iss=2&rank=8}{\nolinkurl{ems-ph.org/issn=0213-223}}


\bibitem{NO}Shahaf Nitzan and Alexander Olevskii, \emph{Revisiting
  Landau's density theorems for Paley-Wiener
  spaces}. C. R. Acad. Sci. Paris 350:9--10 (2012),
  509--512. Available at: \href{http://www.sciencedirect.com/science/article/pii/S1631073X12001392}{\nolinkurl{sciencedirect.com/S1631073X12001392}}









\end{thebibliography}
\end{document}